\documentclass{birkjour}

\usepackage{amssymb, graphicx}
\usepackage{amsmath,amsthm,mathrsfs,amsfonts,bbm}
\usepackage{hyperref}
\usepackage{xcolor}
\usepackage{float}
\usepackage{pgf,tikz}
\usepackage{enumitem}
\usepackage{comment}

\usepackage[noadjust]{cite}

\def \C {\mathbb{C}}

\def \R {\mathbb{R}}
\def \T {\mathbb{T}}
\def \Z {\mathbb{Z}}
\def \N {\mathbb{N}}
\DeclareMathOperator{\supp}{\text{supp}}

\def \ff {\mathcal{F}}
\def\d{\mathrm{d}}

\DeclareFontFamily{U}{mathx}{}
\DeclareFontShape{U}{mathx}{m}{n}{<-> mathx10}{}
\DeclareSymbolFont{mathx}{U}{mathx}{m}{n}
\DeclareMathAccent{\widecheck}{0}{mathx}{"71}

\def\refname {\bf References}

\theoremstyle{plain}

\newtheorem{lemma}{Lemma}[section]

\newtheorem{theorem}[lemma]{Theorem}
\newtheorem{corollary}[lemma]{Corollary}

\theoremstyle{definition}

\newtheorem{example}[lemma]{Example}

\theoremstyle{remark}

\newcommand{\scal}[1]{{\left\langle{#1}\right\rangle}}

\begin{document}
	
	\title[Wiener's Lemma]{On Wiener's Lemma on locally compact\\ abelian groups}
\author[Ph. Jaming]{Philippe Jaming}
	\address{Univ. Bordeaux, IMB, UMR 5251, F-33400 Talence, France. CNRS, IMB, UMR 5251, F-33400 Talence, France.}
	\email{Philippe.Jaming@math.u-bordeaux.fr}

\author[K. Kellay]{Karim Kellay}
\address{Univ. Bordeaux, IMB, UMR 5251, F-33400 Talence, France. CNRS, IMB, UMR 5251, F-33400 Talence, France.} 
\email{k.kellay@math.u-bordeaux.fr}

	\author[R. Perez III]{Rolando Perez III}
	
	\address{R. Perez III, Institute of Mathematics, University of the Philippines Diliman, 1101 Quezon City, Philippines}
	
	\email{rperez@math.upd.edu.ph}
	
	\keywords{Locally compact abelian groups, 
    F\o lner sequence, Wiener's lemma, 
    Bochner-Riesz means}
	
	\subjclass{43A25}

	\begin{abstract}	
    We establish a general form of Wiener’s lemma for measures on locally compact abelian (LCA)  groups by using Fourier analysis and the theory of F\o lner sequences. Our approach provides a unified framework that that encompasses both the discrete and continuous cases.

We also show a version of Wiener's lemma for Bochner-Riesz means on both  $\R^d$ and  $\T^d$.	
	\end{abstract}

	\thanks{The first and second authors are partially supported by the ANR project ANR-24-CE40-5470.}
		
	\maketitle 
	\section{Introduction}
Wiener’s lemma is a well-known identity in harmonic analysis that links the Fourier coefficients of a Borel measure on the circle to its discrete part. In particular, it states that $\mu(\{0\})$ is the limit of the averages
of the Fourier coefficients of $\mu$, given by
\begin{equation}
    \label{eq:wienerlemma}
\mu(\{0\})=\lim_{N\to\infty}\frac{1}{2N+1}\sum_{k=-N}^N\widehat{\mu}(n).
\end{equation}
Using translations, this lemma allows to reconstruct the discrete part of the measure $\mu$ from the partial sums of its Fourier series.

In this work, we investigate Wiener’s lemma in the framework of locally compact abelian (LCA) groups
and show that the discrete part of a measure can be obtained from various averages of its Fourier transform
extending \eqref{eq:wienerlemma}. This result is based on simple Fourier analysis and measure theoretic arguments.
To show the applicability of our results, we
start by showing that the discrete part of a finite measure can be reconstructed by averaging its Fourier transform along a F\o lner sequence. This recovers a previous result of Huang \cite{Hu} in the compact abelian case and of also recovers Wiener's original lemma.
As a second application of our technique, we also show a version of Wiener's lemma for weighted means. In particular, we show that the discrete part  of a finite measure on either $\R^d$ or $\T^d$ can be reconstructed from the limits of means
of its Fourier transform with respect to the Bochner-Riesz kernel.

The structure of the paper is as follows. In Section~\ref{Setting}, we introduce the necessary background on Fourier analysis on LCA groups and define the key notations. Section~\ref{Results} presents our main results, including a general lemma related to Wiener’s lemma and its consequences.
The application to Wiener's lemma with respect
to F\o lner sequences is given in Section ~\ref{Sec:folner}. We conclude with our results on
Bochner-Riesz means on $\R^d$ and on $\T^d$
in Sections~\ref{Sec:BRR} and \ref{Sec:BRT} respectively.

\section{Setting}
\label{Setting}
Throughout, $G$ will be an \text{LCA} group with the Haar measure $m_G$ and $\widehat{G}$ its dual group, {\it i.e.}
the set of characters of $G$.  The unit element of $G$ will be denoted by $\mathbf{1}_G$. Then $\widehat{G}$
is also an LCA group, and we denote by $m_{\widehat{G}}$ its Haar measure.
Depending on the circumstances, it is more convenient to consider elements of $\widehat{G}$
as characters (continuous group   homomorphisms $\gamma: G\to\{z\in\C\ :\ |z|=1\}$) or as a set that parametrizes those functions.
The Fourier transform of $f\in L^1(G)$ is the function on $\widehat{G}$ defined by
$$
\widehat{f}(\gamma)=\ff_G f(\gamma)=c_G\int_G f(x)\overline{\gamma(x)}\,\mbox{d}m_G(x),
\qquad\gamma\in\hat G
$$
where $c_G$ is a normalization constant depending on $G$. 
Recall that if further 
$\widehat{f}\in L^1(\widehat{G})$, then $f=\ff^{-1}_G\widehat{f}$
with the inverse Fourier transform $\ff^{-1}_G$ defined by
$$
\ff^{-1}_Gg(x)=c_{\widehat{G}}\int_{\widehat{G}}g(\gamma)\gamma(x)\,\mbox{d}m_{\widehat{G}}(\gamma),
$$ 
where $g=\hat f$. 
We will always assume that $c_G$ is chosen so that the 
Fourier transform extends to a unitary transform, i.e., $\|\widehat{f}\|_{L^2(\widehat{G})}=\|f\|_{L^2(G)}$ which implies also that $c_G=c_{\widehat{G}}$.
Note that $\ff^{-1}_G g=\overline{\ff_{\widehat G} \bar g}$ 
so that $\ff^{-1}_G$ also extends to a unitary transform on $L^2(\hat G)$
and that $f=\ff^{-1}_G\ff_Gf$ for $f\in L^2(G)$. 

Here we will further extend the Fourier transform to measures: if $\mu$ is a finite complex-valued measure on $G$, then we define
$$
\widehat{\mu}(\gamma)=\ff_G \mu(\gamma)=c_G\int_G \overline{\gamma(x)}\,\mbox{d}\mu(x),
\qquad\gamma\in\hat G.
$$
The key fact for us is the extension of Parseval's identity to measures: if $g\in L^1(\widehat{G})$,
then
$$
\int_{G}\ff_{\widehat{G}}g(x)\,\d\mu(x)
=\int_{\widehat{G}}g(\gamma)\widehat{\mu}(\gamma)\,\mbox{d}m_{\widehat{G}}(\gamma).
$$
For a more in-depth discussion, see, for example, \cite{GM,Rudin62}.

	\section{Results}\label{Results}
\subsection{A General Lemma}

\begin{lemma}
\label{thm:wienerG}
Let $\mu$ be a complex-valued Borel measure on a locally compact abelian group $G$. 
For $R>0$, let $(\varphi_R)_{R>0}$ be a family of continuous functions on $G$ such that
\begin{enumerate}
	\item for every $R>0$, $\|\varphi_R\|_\infty\leq \varphi_R(\mathbf{1}_G)=1$; and
	\item if $x\not=\mathbf{1}_G$, $\varphi_R(x)\to0$ as $R\to\infty$.
\end{enumerate}
Then
$$
\mu(\{\mathbf{1}_G\})=\lim_{R\rightarrow\infty}\int_G \varphi_R (x)\,\mathrm{d}\mu(x).
$$
Furthermore, if $\varphi_R$ is such that there exists  a 
$\psi_R\in L^1(\widehat{G})$ with $\varphi_R=\ff_{\widehat{G}}[\psi_R]$, then
\begin{equation*}
\label{eq:Wiener}
\mu(\{\mathbf{1}_G\})=\lim_{R\rightarrow\infty}\int_{\widehat{G}} \psi_R (\xi)\widehat{\mu}(\xi)\,\mathrm{d}m_{\widehat{G}}(\xi).
\end{equation*}
\end{lemma}
	
\begin{proof}
The second part of the lemma follows directly from the first one and Parseval's identity for measures.

Set $\tilde{\mu}=\mu-\mu(\{\mathbf{1}_G\})\delta_{\mathbf{1}_G}$.	
For a compact neighborhood $V$ of the identity element $\mathbf{1}_G$,
take $\chi_V$ to be a continuous function on $G$ such that $\chi_V=1$ on $V$, 
$0\leq\chi_V\leq 1$, and $\supp\chi_V\subset V^2:=\{v^2\,:v\in V\}$. Writing
$$
\varphi_R=\varphi_R\chi_{V}+\varphi_R(1-\chi_{V}),
$$
we obtain
\begin{align*}
\left| \int_G \varphi_R (x)d\tilde{\mu}(x)\right|
&\leq \int_{V^2} |\varphi_R(x)| d|\tilde\mu|(x)
+ \int_{G\setminus V} |\varphi_R(x)| d|\tilde\mu|(x)\\[2mm]
&\leq |\tilde\mu|(V^2)+\int_{G\setminus V} |\varphi_R(x)| d|\tilde\mu|(x).
\end{align*}
Now, given $\varepsilon>0$, choose $V$ to be a sufficiently small neighborhood of $\mathbf{1}_G$ 
to have $|\tilde\mu|(V^2)\leq\varepsilon$.
Next, since $\varphi_R(x)\to 0$ as $R\rightarrow\infty$ for $x\in G\setminus V$, $|\varphi_R|\leq 1$, and $|\tilde\mu|$
is a finite measure, 
$$
\int_{G\setminus V} |\varphi_R(x)| d|\tilde\mu|(x)\to 0
$$
as $R\rightarrow \infty$ by dominated convergence. We thus choose $R$ sufficiently large so that this quantity is $\leq\varepsilon$ to deduce that
$$\displaystyle\left| \int_G \varphi_R (x)d\tilde{\mu}(x)\right|\leq2\varepsilon.$$
It follows that
$$
\lim_{R\to\infty} \int_G \varphi_R (x)d\tilde{\mu}(x)=0.
$$
Expressing $\tilde\mu$ in terms of $\mu$, we obtain
		 \begin{equation*}
		 	\mu(\{\mathbf{1}_G\})=\lim_{R\rightarrow\infty}\int_G \varphi_R (x)d\mu(x)
		 \end{equation*}
as claimed.
\end{proof}

\subsection{F\o lner sequences}\label{Sec:folner}
Recall that a F\o lner sequence in $\widehat{G}$ is a sequence $(F_n)_{n\geq 1}$ of subsets of $\widehat{G}$ such that,
for every $\gamma\in \widehat{G}$,
$$
\lim_{n\to\infty}\frac{m_{\widehat{G}}\bigl(F_n\Delta(\gamma F_n)\bigr)}{m_{\widehat{G}}(F_n)}=0.
$$
As $\widehat{G}$ is a locally abelian group, it is amenable,  and F\o lner sequences are known to exist
\cite{Pat}.
For instance, when $G=\T^d$, we have $\widehat{G}=\Z^d$, and one may take $F_n=\{-n,\ldots,n\}^d$.
When $G=\R^d$, then $\widehat{G}=\R^d$ and one may take $F_n=[-n,n]^d$.

For each $n\in\mathbb N$ and $x\in G$, consider 
\begin{equation}
\label{eq:folner1}
\varphi_n(x)=\dfrac{\ff_{\widehat{G}}^{-1}[\mathbbm{1}_{F_n}](x)}{\ff_{\widehat{G}}^{-1}[\mathbbm{1}_{F_n}](\mathbf{1}_G)}
=\frac{1}{m_{\widehat{G}}(F_n)}\int_{F_n}\gamma(x)\,\mbox{d}m_{\widehat{G}}(\gamma).
\end{equation}
Obviously
$\|\varphi_R\|_\infty\leq \varphi_R(\mathbf{1}_G)=1$.

Using translation invariance of $m_{\widehat{G}}$, for any $\gamma_0\in\widehat{G}$
we also have
$$
\varphi_n(x)=\frac{1}{m_{\widehat{G}}(F_n)}\int_{\gamma_0 F_n}(\gamma_0^{-1}\gamma)(x)\,\mbox{d}m_{\widehat{G}}(\gamma),
$$
that is,
\begin{equation}
\label{eq:folner2}
\gamma_0(x)\varphi_n(x)=\frac{1}{m_{\widehat{G}}(F_n)}\int_{\gamma_0 F_n}\gamma(x)\,\mbox{d}m_{\widehat{G}}(\gamma).
\end{equation}
Subtracting \eqref{eq:folner2} from \eqref{eq:folner1}, we obtain
$$
\bigl(1-\gamma_0(x))\varphi_n(x)=\frac{1}{m_{\widehat{G}}(F_n)}\int_{\widehat{G}}
\bigl(\mathbbm{1}_{F_n}(\gamma)-\mathbbm{1}_{\gamma_0 F_n}(\gamma)\bigr)
\gamma(x)\,\mbox{d}m_{\widehat{G}}(\gamma).
$$
Now, fix $x\not=\mathbf{1}_G$ and take $\gamma_0\in \widehat{G}$ 
such that $\alpha(x)=|\gamma_0(x)-1|>0$. But then
\begin{align*}
|\varphi_n(x)|&=
\frac{|\bigl(1-\gamma_0(x))\varphi_n(x)|}{\alpha(x)}\\
&\leq\frac{1}{\alpha(x)}\frac{1}{m_{\widehat{G}}(F_n)}\int_{\widehat{G}}
\bigl|\mathbbm{1}_{F_n}(\gamma)-\mathbbm{1}_{\gamma_0 F_n}(\gamma)\bigr|
|\gamma(x)|\,\mbox{d}m_{\widehat{G}}(\gamma)\\
&=\frac{1}{\alpha(x)}\frac{m_{\widehat{G}}\bigl(F_n\Delta(\gamma_0 F_n)\bigr)}{m_{\widehat{G}}(F_n)}
\to 0
\end{align*}
as $n\rightarrow\infty$. Applying Lemma \ref{thm:wienerG}, we deduce that
$$
\mu(\{\mathbf{1}_G\})=\lim_{n\rightarrow\infty}\frac{1}{m_{\widehat{G}}(F_n)}\int_{F_n}\widehat{\mu}(\gamma)\,\mathrm{d}m_{\widehat{G}}(\gamma).
$$
Finally, replacing $\mu$ with its translate $\mu_x(E)=\mu(x E)$, we obtain the following:

\begin{theorem}[Wiener's lemma for F\o lner sequences]
Let $\mu$ be a complex-valued Borel measure on a locally compact abelian group $G$. 
Let $(F_n)_{n\geq 1}$ be a F\o lner sequence in the dual group $\widehat{G}$.
Then, for every $x\in G$,
$$
\mu(\{x\})=\lim_{n\rightarrow\infty}\frac{1}{m_{\widehat{G}}(F_n)}\int_{F_n}
\gamma(x)\widehat{\mu}(\gamma)\,\mathrm{d}m_{\widehat{G}}(\gamma).
$$
\end{theorem}
We now turn to the following examples.
\begin{example}
For $x\in G=\T^d$, we have
$$
\mu(\{x\})=\lim_{n\rightarrow\infty}\frac{1}{(2n+1)^d}\sum_{j\in\{-n,\ldots,n\}^d}
\widehat{\mu}(j)e^{2i\pi \scal{j,x}}
$$
as well as
$$
\mu(\{x\})=\lim_{n\rightarrow\infty}\frac{1}{N_d(n)}\sum_{j\in\Z^d,|j|\leq n}
\widehat{\mu}(j)e^{2i\pi \scal{j,x}},
$$
where $N_d(n)=|\Z^d\cap \overline{B(0,n)}|$ where $\overline{B(0,n)}$ is the closed ball of radius $n$ in $\R^d$
for the Euclidean norm $|\cdot|$.

On the other hand, for $x\in G=\R^d$, we have
$$
\mu(\{x\})=\lim_{R\rightarrow\infty}\frac{1}{(2R)^d}\int_{[-R,R]^d}
\widehat{\mu}(\xi)e^{2i\pi \scal{x,\xi}}\,\mbox{d}\xi
$$
and
$$
\mu(\{x\})=\lim_{R\rightarrow\infty}\frac{1}{\omega_dR^d}\int_{B(0,R)}
\widehat{\mu}(\xi)e^{2i\pi \scal{x,\xi}}\,\mbox{d}\xi,
$$
where $\omega_d$ is the volume of the unit ball.
\end{example}

\begin{example}
More generally, it was shown by Day \cite{Day}
that a sequence $(C_n)_{n\geq 1}$ of convex subsets of $\R^d$
is a F\o lner sequence if and only if
the inner radius $\rho(C_n)\to\infty$. Recall that $\rho(C_n)$ is the supremum
of the radii of balls contained in $C_n$.
In particular, if $(C_n)_{n\geq 1}$ is a sequence of convex subsets of $\R^d$ with $\rho(C_n)\to\infty$, then
for every finite measure $\mu$ and every $x\in\R^d$, we have
$$
\mu(\{x\})=\lim_{R\rightarrow\infty}\frac{1}{|C_n|}\int_{C_n}
\widehat{\mu}(\xi)e^{2i\pi \scal{x,\xi}}\,\mbox{d}\xi.
$$
\end{example}

\subsection{Weighted means on $\R^d$}\label{Sec:BRR}

We focus on  further applications of Lemma \ref{thm:wienerG}
when $G=\R^d$. In this case, we will use scaling to construct $\varphi_R$:

	\begin{lemma}\label{lem:scale}
	Let $\psi\in L^1(\mathbb R^d)$ be nonnegative. Then for $R>0$, the function $\varphi_R(\xi)=\dfrac{\widehat{\psi}(R\xi)}{\widehat\psi(0)}$ satisfies
	
		\begin{enumerate}
			\item $\|\varphi_R\|_\infty\leq1$; and
			\item for any $\eta>0$, $\displaystyle \|\varphi_R\chi_{\mathbb R^d\setminus B(0,\eta)}\|_\infty\rightarrow 0$ as $R\rightarrow\infty$.
		\end{enumerate}
	\end{lemma}

	\begin{proof}
		Since $\widehat\psi$ is positive-definite, $|\widehat \psi (R\xi)|\leq \widehat\psi(0)$ for all $\xi\in \mathbb R^d$ and so $\|\varphi_R\|_\infty\leq 1$. The second property follows from the fact that $\widehat \psi\in C_0 (\mathbb R^d)$. 
	\end{proof}
	
	As an immediate consequence of Lemma \ref{thm:wienerG} when $G=\mathbb R^d$:
	\begin{corollary}
		\label{cor:wiener-rd}
		Let $\mu$ be a complex-valued Borel measure on $\mathbb R^d$ and let $x\in \mathbb R^d$. If $\psi\in L^1(\mathbb R^d)$ is nonnegative, then 
		\begin{align*}
		\mu(\{x\})&=\lim_{R\rightarrow\infty}\dfrac{1}{R^d\,\widehat\psi(0)}\int_{\mathbb R^d} \psi\left(\dfrac{\xi}{R}\right)\widehat\mu(\xi)e^{2\pi i\langle x,\xi\rangle}\,\mathrm{d}\xi\\[2mm]
		&=\lim_{R\rightarrow\infty}\dfrac{1}{\widehat\psi(0)}\int_{\mathbb R^d}\widehat\psi(R(t-x))\,\mathrm{d}\mu(t).
		\end{align*}	
	\end{corollary}
    
	\begin{proof}
		If $\psi\in L^1(\mathbb R^d)$ is nonnegative, combining Lemmas \ref{thm:wienerG} and \ref{lem:scale}
        gives
		\begin{equation*}
			\label{eq:cor-wiener-rd-proof1}
			\mu(\{0\})=\lim_{R\rightarrow\infty}\dfrac{1}{\widehat\psi(0)}\int_{\mathbb R^d}\widehat\psi(R\xi)\,\mathrm{d}\mu(\xi).
		\end{equation*}
Applying this for the translated measure $\mu_x$ so that $\widehat{\mu_x}(\xi)=e^{2\pi i \langle x,\xi\rangle}\widehat\mu(\xi)$, we obtain the second formula. The first one
follows from Parseval's identity.
%
	\end{proof}
	
	Let us illustrate the result using some examples.
	
	\begin{example}
		For $\xi\in \mathbb R^d$, consider the Gaussian $\psi(\xi)=e^{-\pi |\xi|^2}=\widehat\psi(\xi)$. By Corollary \ref{cor:wiener-rd}, we obtain 
		\begin{align*}
		\mu(\{x\})&=\lim_{R\rightarrow\infty}\dfrac{1}{R^d}\int_{\mathbb R^d} e^{-\pi \frac{|\xi|^2}{R^2}}e^{2\pi i\langle x,\xi\rangle}\,\widehat\mu(\xi)\,\mathrm{d}\xi\\
		&=\lim_{R\rightarrow\infty}\int_{\mathbb R^d}e^{-\pi R^2|t-x|^2}\,\mathrm{d}\mu(t)
		\end{align*}
		for all $x\in\mathbb R^d$.
	
	\end{example}
	
	
	\begin{example}\label{eq:bochnerriezR}
		For $\alpha>0$ and $\xi\in \mathbb R^d$, define
		$$
		m_{\alpha}(\xi):=(1-|\xi|^2)_+^\alpha
		$$
		where $(1-|\xi|^2)_+^\alpha$ gives the positive part of $(1-|\xi|^2)^\alpha$.
		From \cite[B.5]{G2014}, it is known that
		$$
		\widehat{m_\alpha}(\xi)=2^{\alpha}(2\pi)^{\frac{d}{2}}\Gamma(\alpha+1)\dfrac{J_{\frac{d}{2}+\alpha}(2\pi |\xi|)}{(2\pi |\xi|)^{\frac{d}{2}+\alpha}},\qquad \xi\in \mathbb R^d
		$$
		where $J_{\frac{d}{2}+\alpha}$ is the Bessel function of order $\frac{d}{2}+\alpha$. Letting $|\xi|\rightarrow 0$, we get
		$$m_\alpha(0)=\dfrac{2^{\alpha}(2\pi)^{\frac{d}{2}}\Gamma(\alpha+1)}{2^{\frac{d}{2}+\alpha}\Gamma(\frac{d}{2}+\alpha+1)}=\dfrac{\pi^{\frac{d}{2}}\Gamma(\alpha+1)}{\Gamma(\frac{d}{2}+\alpha+1)}.$$ 
		By Corollary \ref{cor:wiener-rd}, we obtain
		\begin{align*}
		\mu(\{x\})&=\lim_{R\rightarrow\infty}\dfrac{1}{R^d\,\widehat{m_\alpha}(0)}\int_{\mathbb R^d} {m_\alpha}\left(\dfrac{\xi}{R}\right)\widehat\mu(\xi)e^{2\pi i\langle x,\xi\rangle}\,\mathrm{d}\xi\\
		&=\lim_{R\rightarrow \infty}\dfrac{\Gamma(\frac{d}{2}+\alpha+1)}{R^d\pi^{\frac{d}{2}}\Gamma(\alpha+1)}\int_{\mathbb R^d}\left(1-\dfrac{|\xi|^2}{R^2}\right)_+^\alpha e^{2\pi i x\xi}\widehat\mu(\xi)\,\mathrm{d}\xi
		\end{align*}
		for all $x\in\mathbb R^d$. Note that if the dimension $d$ is even, the gamma factors outside the integral can be simplified further.
	\end{example}

\subsection{Weighted means on $\T^d$}\label{Sec:BRT}

In this section, we are going to prove an analogue of Example \ref{eq:bochnerriezR} on $\T^d$.
The difficulty is that pointwise estimates of Bochner-Riesz on $\T^d$ are not available as
results on $\T^d$ are usually obtained by transference from results on $\R^d$.
Here we define, for $\delta> 0,\,N\in\mathbb N$, and $x\in\mathbb \T^d$,
$$
B_N^{d,\delta}(x)=\sum_{k\in\Z^d}\left(1-\frac{|k|^2}{N^2}\right)_+^\delta e^{2i\pi\scal{k,x}}
$$
where, as usual, $(1-t)^\delta_+$ is $(1-t)^\delta$ for $0\leq t\leq 1$ and $0$ otherwise.
The spherical Dirichlet kernel is then $$D_N^d(x)=B_N^{d,0}(x)=\displaystyle \sum_{k\in\Z^d, |k|\leq N}e^{2i\pi\scal{k,x}}.$$
Moreover, set
$$
\beta_N^{d,\delta}=B_N^{d,\delta}(0)=\sum_{k\in\Z^d, |k|\leq N}\left(1-\frac{|k|^2}{N^2}\right)_+^\delta.
$$
Note that
$$
\beta_N^{d,\delta}\geq \sum_{k\in\Z^d, |k|\leq N/2}\left(1-\frac{|k|^2}{N^2}\right)_+^\delta
\gtrsim N^d.
$$

Our aim is to prove the following:

\begin{theorem}
    Let $\mu$ be a finite measure on $\T^d$ and $\delta\geq 0$.
    Then, for every $x\in\T^d$,
    $$
\mu(\{x\})=\lim_{N\rightarrow\infty}\frac{1}{\beta_N^{d,\delta}}\sum_{j\in\Z^d,|j|\leq N}
\left(1-\frac{|j|^2}{N^2}\right)_+^\delta\widehat{\mu}(j)e^{2i\pi \scal{j,x}}.
$$
\end{theorem}

\begin{proof}
Translating the measure, it is enough to prove the result for $x=0$.
It will follow from Lemma \ref{thm:wienerG} by considering $\varphi_N(x)=\dfrac{B_N^{d,\delta}(x)}{\beta_N^{d,\delta}}$. Obviously, $|\varphi_N(x)|\leq\varphi_N(0)=1$, so the only thing
we need to show is that when $x\notin\Z^d$, $\varphi_N(x)\to 0$ as $N\rightarrow \infty$.

We have not been able to find a convenient reference containing an estimate of $B_N^{d,\delta}$
where this is done, except for $d=1$ where an estimate can be found in several textbooks
including \cite{Zyg}. For sake of completeness, we will give a simple
estimate, and include a proof for $d=1$ as well.

We will start with estimating the spherical Dirichlet kernel and prove, by induction on
the dimension, that for $x\notin\Z^d$, $D_N^d(x)=O(N^{d-1})$.

Indeed, when $d=1$ and $x\notin\Z$,
$D_N^1(x)=\dfrac{\sin (2N+1)\pi x}{\sin\pi x}$ thus $|D_N^1(x)|\leq \dfrac{1}{|\sin\pi x|}$
is bounded independently of $N$. Let us now assume that the estimate is proven
in dimension $d$ and take $x=(x_1,\ldots,x_{d+1})\in\R^{d+1}\setminus\Z^{d+1}$. Then, up to a permutation
of the coordinates, we may assume that $\bar x:=(x_2,\ldots,x_{d+1})\notin\Z^d$.
Then write
\begin{align*}
D_N^{d+1}(x)&=\sum_{j=-N}^{N}e^{2i\pi jx_1}\sum_{k\in\Z^d,|k|\leq \sqrt{N^2-j^2}}e^{2i\pi\scal{k,\bar x}}\\
&=\sum_{j=-N}^{N}e^{2i\pi jx_1}D^d_{[\sqrt{N^2-j^2}]}(\bar x).
\end{align*}
From the induction hypothesis $|D^d_{[\sqrt{N^2-j^2}]}(\bar x)|\lesssim [\sqrt{N^2-j^2}]^d\lesssim N^{d-1}$
from which we immediately obtain that $|D_N^{d+1}(x)|\lesssim N^d$.

\smallskip

We now estimate the Bochner-Riesz kernel.
For $j\in\N$, write $S_j=\{k\in\Z^d\,:|k|^2=j\}$ for the intersection of $\Z^d$ with the 
sphere of radius $j$ (which may be empty). Then
$$
B_N^{d,\delta}(x)=\sum_{j=0}^{N^2}\left(1-\frac{j}{N^2}\right)_+^\delta\sum_{k\in S_j}e^{2i\pi\scal{k,x}}.
$$
Using Abel summation, we may then rewrite this as
$$
B_N^{d,\delta}(x)=1-\left(1-\frac{1}{N^2}\right)^\delta
+\sum_{j=1}^{N^2-1}\left[\left(1-\frac{j}{N^2}\right)_+^\delta-\left(1-\frac{j+1}{N^2}\right)_+^\delta\right]
D_j^d(x).
$$
Using our estimate for the Dirichlet kernel and the Mean Value Theorem,
we obtain that when
$x\notin\Z^d$,
\begin{align*}
|B_N^{d,\delta}(x)|&\lesssim
N^{d-1}\sum_{j=0}^{N^2}\frac{\delta}{N^2}\left(1-\frac{j}{N^2}\right)_+^{\delta-1}
\\
&\lesssim N^{d-1}\int_0^1(1-x)^{\delta-1}\,\mbox{d}x\lesssim N^{d-1}.
\end{align*}
As $\beta_N^{d,\delta}\gtrsim N^d$, we obtain $|\varphi_N(x)|\lesssim N^{-1}$ when $x\notin\Z^d$
as claimed.
\end{proof}

\nocite{*}
\bibliographystyle{plain}

\end{document}